\documentclass[11pt, a4paper]{amsart}
\usepackage{amssymb}
\usepackage{amsmath}
\usepackage{amssymb}
\usepackage{amsthm}
\usepackage{bbm}
\usepackage{bm}
\usepackage{mathtools}
\usepackage{mathrsfs}
\usepackage[T1]{fontenc}
\usepackage[usenames,dvipsnames,svgnames,table]{xcolor}
\usepackage{esint}
\allowdisplaybreaks
\usepackage{combelow}
\usepackage{enumitem}
\setenumerate{label={\rm (\alph{*})}}
\usepackage[colorlinks=true,citecolor=black,urlcolor=black,
linkcolor=black]{hyperref}

\usepackage{graphicx}
\usepackage{tikz}

\usepackage{geometry}
\makeatletter
\newcommand*\bigcdot{\mathpalette\bigcdot@{.5}}
\newcommand*\bigcdot@[2]{\mathbin{\vcenter{\hbox{\scalebox{#2}{$\m@th#1\bullet$}}}}}
\makeatother

\newtheorem{theorem}{Theorem}%[section]
\newtheorem{lemma}[theorem]{Lemma}
\newtheorem{proposition}[theorem]{Proposition}

\theoremstyle{remark}

\usepackage{booktabs}

\normalsize

\def\XXint#1#2#3{{\setbox0=\hbox{$#1{#2#3}{\int}$ }
		\vcenter{\hbox{$#2#3$ }}\kern-.6\wd0}}

\newcommand{\R}{\mathbb{R}}
\newcommand{\B}{\mathscr{B}}

\newcommand{\sobo}{\operatorname{W}}

\newcommand{\lebe}{\operatorname{L}}

\renewcommand{\leq}{\leqslant}

\newcommand{\sym}{\operatorname{sym}}

\usepackage{tikz-cd}
\newcommand{\proj}{\mathrm{Proj}}
\newcommand{\rmim}{\mathrm{im\,}}

\newcommand{\A}{\mathscr{A}}

\AtEndDocument{\bigskip{\footnotesize%
		
		\vspace{2.4pt}
		\noindent\textsc{Centro di Ricerca Matematica Ennio de Giorgi, Scuola Normale Superiore,
			\\Piazza dei Cavalieri, 3, 56126 Pisa, Italy, }   \\
		\noindent\textit{E-mail address}, B.~Rai\cb{t}\u{a}: \texttt{bogdanraita@gmail.com}
}}
\begin{document}
	\title[Potential operators for compensated compactness]{A simple construction of potential operators for compensated compactness}
	\author[B. Rai\cb{t}\u{a}]{Bogdan Rai\cb{t}\u{a}}
	%\address[B.~Rai\cb{t}\u{a}]{Max-Planck-Institut f\"ur Mathematik in den Naturwissenschaften, Inselstra{\ss}e 22, 04103, Leipzig, Germany
		%}

	%\date{\today}

	\begin{abstract}
		We give a short proof of the fact that each homogeneous linear differential operator $\A$ of constant rank admits a homogeneous potential operator $\B$, meaning that $$\ker\A(\xi)=\mathrm{im\,}\B(\xi) \quad\text{for }\xi\in\R^n\setminus\{0\}.$$
		We make some refinements of the original result and some related remarks.
	\end{abstract}
	\maketitle
	%\vspace{-7pt}
	The framework of compensated compactness was introduced in \cite{Tartar1,Tartar2,Murat1,Murat2}  as a unified framework for continuum mechanics. It was subsequently developed in many directions and has received numerous substantial contributions \cite{BCO,Dacorogna82,CLMS,FM99,Tartar_survey}. The theme of compensated compactness is the rigid interaction between nonlinear effects (constitutive relations) and weakly convergent sequences that satisfy an additional linear differential constraint (balance relations).
	Typically, the linear equations satisfy the so called \emph{constant rank} condition \cite{SW,Murat2} that we will recall bellow. Under this assumption, it was recently proved in \cite{R_pot} (cf. \cite{ARS}) that the system of balance relations can be solved efficiently under suitable boundary conditions. This implies, for instance, that the notion of $\A$-quasiconvexity of Fonseca--M\"uller \cite{FM99} coincides with the seemingly more restrictive notion of $\A$-$\B$-quasiconvexity of Dacorogna \cite{Dacorogna82}. Further advances were obtained in \cite{GR,GKR,GRS,AAR,KR}. 
	
	For the purposes of this note, we hardly need to refer to vectorial differential operators, but rather think of them as matrix valued polynomials
	$$
	\A(\xi)\coloneqq (\A_{ij}(\xi))_{1\leq i\leq m,\,1\leq j\leq N}
	\quad\text{for }\xi\in\R^n,
	$$
	that can be identified with a vectorial differential operator with constant coefficients. Let
	$$
	\mathcal R_\A \coloneqq \{\xi\in\R^n\colon \mathrm{rank\,}\A(\xi) \text{ is maximal}\},
	$$
	which is an open dense subset of $\R^n$ (its complement is a proper real variety). 
	
	The result of \cite[Thm.~1]{R_pot} is a corollary of  the following:
	\begin{theorem}\label{thm:main}
		Let $\A$ be \emph{any} $\R^{M\times N}$ valued polynomial. Then there exists a $\R^{N\times N}$ valued polynomial $\B$ such that
		$$
		\ker\A (\xi)=\rmim\B(\xi)\quad \text{for }\xi\in\mathcal R_\A.
		$$
		Moreover, if $\A$ has homogeneous rows, then $\B$ can be chosen to have the same homogeneity in every entry: if $\A_{ij}$ is  homogeneous of degree $h_i$ and $r$ is the generic rank of $\A$, we can choose $\B$ to be $(2r\max_i{h_i})$-homogeneous.
	\end{theorem}
	Since $\mathcal R_\A$ is open, we can infer in addition that $\A\B\equiv 0$, so
	$$
	\ker\A (\xi)\supset\rmim\B(\xi)\quad \text{for }\xi\in\mathcal \R^n.
	$$
	The inclusion must be strict on the entire $\R^n\setminus\mathcal R_\A$; an instructive example is the wave operator $\A(\xi_1,\xi_2)=\xi_1^2-\xi_2^2$.  The proof of Theorem~\ref{thm:main} follows easily from the following:
	\begin{lemma}\label{lem:proj}
		Let $M\in\R^{N\times N}_{\sym}$ and $Q(t)\coloneqq(t-\lambda_1)\ldots(t-\lambda_r)=\sum_{i=0}^r a_{r-i}t^i$, where $\lambda_1,\ldots\lambda_r$ are the nonzero eigenvalues of $M$, counting multiplicities. Then
		$$
		\proj_{\ker M}=\frac{1}{a_r}\,Q(M).
		$$
	\end{lemma}
	\begin{proof}[Proof of Lemma~\ref{lem:proj}]
		Consider an orthogonal matrix $O\in\R^{N\times N}$ such that $M=O^*DO$ where $D=\mathrm{diag}(\lambda_1,\ldots,\lambda_r,0,\ldots 0)$. Then we can compute
		$$
		Q(M)=O^*Q(D)O=O^* 
		\left(\begin{matrix}
			\mathbf{0}_{r\times r}&\mathbf{0}_{r\times (N-r)}\\
			\mathbf{0}_{(N-r)\times r}&(-1)^r\lambda_1\ldots \lambda_r\mathbf{I}_{N-r}
		\end{matrix}\right)
		O=a_rO^*\proj_{\ker D}O.
		$$
		One concludes by checking the elementary fact that $\proj_{\ker O^*DO}=O^*\proj_{\ker D}O$.
	\end{proof}
	\begin{proof}[Proof of Theorem \ref{thm:main}]
		The proof of the exact relation follows at once from Lemma \ref{lem:proj} by setting $M=\A^*(\xi)\A(\xi)$ and $\B(\xi)=Q(M)$. In the case of homogeneous operators, we will make a slight modification. We can assume that $h_1\leq h_2\leq \ldots \leq h_m$. We define the  operator $\tilde \A$ row wise by $\tilde \A_{ij}(\xi)v\coloneqq [\A_{ij}(\xi)v]\otimes\xi^{\otimes(h_m-h_i)}$ for $v\in\R^N$, so basically $\tilde\A$ is defined as $\tilde\A_{ij}=D^{h_m-h_i}\A_{ij}$. We have that $\tilde \A$ is $h_m$-homogeneous in every entry and $\ker\A(\xi)=\ker\tilde\A(\xi)$ for all $\xi\in\R^n$. Setting $M=\tilde\A^*(\xi)\tilde\A(\xi)$ we remark that $M$ is an $\R^{N\times N}_{\sym}$-valued $2h_m$-homogeneous polynomial, so we can choose $\B(\xi)=Q(M)$.
	\end{proof}
	We can also use Lemma~2 to prove the main result of \cite{Decell}, which was originally used to prove \cite[Thm.~1]{R_pot}. We will aim to give a formula for the Moore--Penrose generalized inverse: for any matrix $P\in\R^{m\times N}$ there is a unique matrix $P^\dagger\in\R^{N\times m}$ such that:
	$$
	P^\dagger P P^\dagger = P^\dagger,\quad PP^\dagger P=P, \quad (PP^\dagger)^*=PP^\dagger,\quad (P^\dagger P)^*=P^\dagger P,
	$$
	or, equivalently, if
	$$
	PP^\dagger = \proj_{\rmim P},\quad P^\dagger P=\proj_{\rmim P^*}.
	$$
	For a proof of these facts and general properties of pseudoinverses, see \cite{Campbell}. We can now prove a computable formula for the generalized inverse:
	\begin{theorem}[Decell]\label{thm:decell}
		Let $P\in\R^{m\times N}$ and let $a_i$ be given by Lemma~\ref{lem:proj} for $M= PP^*$. Then
		$$
		P^\dagger=-\frac{1}{a_r}\sum_{i=1}^r a_{r-i}P^*(PP^*)^{i-1}.
		$$
	\end{theorem}
	This result, or rather the fact that under the constant rank condition, it implies that the multiplier
	$
	\xi\mapsto\B^\dagger(\xi)
	$
	is smooth away from zero, is still needed to infer that the partial differential equation $\A v=0$ can be solved for $v=\B u$, e.g., with periodic boundary conditions, cf. \cite[Lem.~2]{R_pot}. We will also use this fact in Proposition \ref{prop:hhhh} to deduce a variant of the standard Helmholtz decomposition that is often used in the study of these problems. 
	\begin{proof}[Proof of Theorem~\ref{thm:decell}]
		By Lemma~\ref{lem:proj} we have that 
		$$
		\proj_{\ker M}=\frac{Q(M)}{a_r}=\mathbf{I}_m+\frac{1}{a_r}\sum_{i=1}^ra_{r-i}M^i, 
		$$
		so
		$$
		\proj_{\rmim M}=-\frac{1}{a_r}\sum_{i=1}^ra_{r-i}M^i.
		$$
		Therefore, by using both definitions of the generalized inverse, we have
		\begin{align*}
			P^\dagger&=P^\dagger P P^\dagger =P^\dagger \proj_{\rmim M}=-\frac{1}{a_r}\sum_{i=1}^ra_{r-i}P^\dagger (PP^*)^i=-\frac{1}{a_r}\sum_{i=1}^ra_{r-i}P^\dagger PP^*(PP^*)^{i-1},
		\end{align*}
		which is enough to conclude.
	\end{proof}
	A homogeneous differential operator $\A$ is said to have \textbf{constant rank} if $\mathcal R_\A=\{0\}$. Here by homogeneous differential operator we mean that each row $(\A_{ij})_j$ is $h_i$-homogeneous for some natural number $h_i$. In this case, considering the differential system $\A v=0$ (compensating condition), we have that each equation is homogeneous, of various degrees.
	\begin{proposition}\label{prop:hhhh}
		Let $1<p<\infty$, $\A$ be homogeneous (in the rows) and of constant rank, and $\B$ be as in Theorem \ref{thm:main}; say that $\B$ has degree $k$. Then there exists a bounded linear map $v\in\lebe^p(\R^n,\R^N)\mapsto u\in{\dot\sobo}{^{k,p}}(\R^n,\R^N)$ such that
		$$
		\|D^k u\|_{\lebe^p}\leq c\|\B u\|_{\lebe^p}\leq c\|v\|_{\lebe^p}\quad\text{and}\quad\|v-\B u\|_{\lebe^p}\leq c \sum_{i=1}^m\|\A_iv\|_{{\dot\sobo}{^{-h_i,p}}},
		$$
		where $h_i$ is the homogeneity of the $i$th row of $\A$.
	\end{proposition}
	One can formulate a variant of the bound for $u$ also in the case in which $\B$ has homogeneous columns, say of degree $\tilde h_j$. In that case, one has the estimates
	\begin{align}\label{eq:B}
		\|D^{\tilde h_j}u_j\|_{\lebe^p}\leq c\|\B u\|_{\lebe^p}.
	\end{align}
	We leave this exercise for the reader.
	\begin{proof}[Proof of Proposition \ref{prop:hhhh}]
		We write $v=v_1+v_2$, where $v_1,\,v_2$ are defined by the multipliers
		$$
		\hat v_1(\xi)\coloneqq \proj_{\ker \A^*(\xi)\A(\xi)}\hat v(\xi)\quad\text{and}\quad \hat v_2(\xi)=\A^\dagger(\xi) \widehat{\A v}(\xi).
		$$
		By Lemma \ref{lem:proj} and the constant rank condition, we have that $v_1$ is defined by a convolution with a  0-homogeneous multiplier that is smooth in $\R^n\setminus\{0\}$. By the H\"ormander--Mikhlin Theorem, we have that
		$$
		\|v_1\|_{\lebe^p}\leq c\|v\|_{\lebe^p}.
		$$
		Moreover, $\A v_1=0$, so that we can write $v_1=\B u$ for $u\in\dot\sobo{^{k,p}}$ with the required bound by using the ideas in \cite[Sec. 3.3]{GR}. 
		
		Writing $C_i(\xi)$ for the  columns of $\A^\dagger(\xi)$, it is a simple exercise to check that $C_i$ is $(-h_i)$-homogeneous. We have that
		$$
		\hat v_2(\xi)=\sum_{i=1}^M C_i(\xi) \widehat{\A_i v}(\xi)=\sum_{i=1}^M C_i\left(\frac{\xi}{|\xi|}\right) \frac{\widehat{\A_i v}(\xi)}{|\xi|^{h_i}}.
		$$
		By Theorem \ref{thm:decell} and the constant rank condition we can infer the necessary smoothness to apply the H\"ormander--Mikhlin Theorem again to obtain
		$$
		\|v_2\|_{\lebe^p}\leq c \sum_{i=1}^M\|\A_iv\|_{{\dot\sobo}{^{-h_i,p}}}.
		$$
		The proof is complete.
	\end{proof}
	\begin{proof}[Sketch proof of \eqref{eq:B}]
		We write $R_j(\xi)$ for the rows of $\B^\dagger(\xi)$, which are $(-\tilde h_j)$-homogeneous. The calculations in the proof of Proposition~\ref{prop:hhhh} are replaced by
		$$
		\widehat{D^{\tilde h_j}u_j}(\xi)=R_j(\xi)\widehat{\B u}(\xi)\otimes\xi^{\otimes \tilde h_j}=R_j\left(\frac{\xi}{|\xi|}\right)\widehat{\B u}(\xi)\otimes\left(\frac{\xi}{|\xi|}\right)^{\otimes\tilde h_j},
		$$
		and the H\"ormander--Mikhlin Theorem applies here as well.
	\end{proof}
	

\begin{thebibliography}{99}
		\bibitem{AAR} Arroyo-Rabasa, A., 2021. Characterization of Generalized Young Measures Generated by ${\mathcal {A}}$-free Measures. Archive for Rational Mechanics and Analysis, 242(1), pp.235-325.
		
		\bibitem{ARS} Arroyo-Rabasa, A. and Simental, J., 2021. An elementary proof of the homological properties of constant-rank operators. arXiv preprint arXiv:2107.05098.
		
		\bibitem{BCO} Ball, J.M., Currie, J.C. and Olver, P.J., 1981. Null Lagrangians, weak continuity, and variational problems of arbitrary order. Journal of Functional Analysis, 41(2), pp.135-174.
		
		\bibitem{Campbell} Campbell, S.L. and Meyer, C.D., 2009. Generalized inverses of linear transformations. Society for industrial and applied Mathematics.
		
		\bibitem{CLMS} Coifman, R., Lions, P.L., Meyer, Y. and Semmes, S., 1993. Compensated compactness and Hardy spaces. Journal de mathématiques pures et appliquées, 72(3), pp.247-286.
		
		\bibitem{Dacorogna82} Dacorogna, B., 1982. Quasi-convexité et semi-continuité inférieure faible des fonctionnelles non linéaires. Annali della Scuola Normale Superiore di Pisa-Classe di Scienze, 9(4), pp.627-644.
		
		\bibitem{Decell} Decell, Jr, H.P., 1965. An application of the Cayley-Hamilton theorem to generalized matrix inversion. SIAM review, 7(4), pp.526-528.
		
		\bibitem{FM99} Fonseca, I. and Müller, S., 1999. $\mathcal A$-Quasiconvexity, Lower Semicontinuity, and Young Measures. SIAM journal on mathematical analysis, 30(6), pp.1355-1390.
		
		\bibitem{GKR} Guerra, A., Kristensen, J., and Rai\cb{t}\u a, B., 2020. Oscillation and concentration under constant rank constraints.  Technical Report OxPDE-20.18, University of Oxford.
		
		\bibitem{GR} Guerra, A. and Rai\cb{t}ă, B., 2019. Quasiconvexity, null Lagrangians, and Hardy space integrability under constant rank constraints. arXiv preprint arXiv:1909.03923.
		
		\bibitem{GRS} Guerra, A., Rai\cb{t}ă, B. and Schrecker, M.R., 2020. Compensated compactness: continuity in optimal weak topologies. arXiv preprint arXiv:2007.00564.
		
		\bibitem{KR} Kristensen, J. and Rai\cb{t}ă, B., 2019. Oscillation and concentration in sequences of PDE constrained measures. arXiv preprint arXiv:1912.09190.
		
		\bibitem{Murat1} Murat, F., 1978. Compacité par compensation. Annali della Scuola Normale Superiore di Pisa-Classe di Scienze, 5(3), pp.489-507.
		
		\bibitem{Murat2} Murat, F., 1981. Compacité par compensation: condition nécessaire et suffisante de continuité faible sous une hypothese de rang constant. Annali della Scuola Normale Superiore di Pisa-Classe di Scienze, 8(1), pp.69-102.
		
		\bibitem{R_pot} Rai\cb{t}ă, B., 2019. Potentials for $\mathcal {A}$-quasiconvexity. Calculus of Variations and Partial Differential Equations, 58(3), pp.1-16.
		
		\bibitem{SW} Schulenberger, J.R. and Wilcox, C.H., 1971. Coerciveness inequalities for nonelliptic systems of partial differential equations. Annali di Matematica Pura ed Applicata, 88(1), pp.229-305.
		
		\bibitem{Tartar1} Tartar, L., 1979. Compensated compactness and applications to partial differential equations. In Nonlinear analysis and mechanics: Heriot-Watt symposium (Vol. 4, pp. 136-212).
		
		\bibitem{Tartar2} Tartar, L., 1983. The compensated compactness method applied to systems of conservation laws. In Systems of nonlinear partial differential equations (pp. 263-285). Springer, Dordrecht.
		
		\bibitem{Tartar_survey} Tartar, L., 2005. Compensation effects in partial differential equations. Rend Accad Naz Sci XL Mem Mat Appl (5), 29(1), pp.395-453.
		
	\end{thebibliography}
\end{document}